\documentclass[12pt]{article}

\usepackage[paper=letterpaper,margin=0.7in,twoside=false,includehead]{geometry}

\usepackage{amsfonts,
	amsmath,
	amsthm,
	amssymb,
	fancyhdr,
	graphicx,
	hyperref,
	mathtools,
	verbatim,
	xcolor,
	xparse
	} 
 
\usepackage[capitalize]{cleveref}
\usepackage[shortlabels]{enumitem}
\usepackage[normalem]{ulem}

\SetEnumitemKey{thmparts}{topsep=2pt plus 1pt minus1pt, itemsep=1pt plus0.5pt minus0.5pt}

\newcommand{\N}{\mathbb{N}}

\newcommand{\R}{\mathbb{R}}

\newcommand{\C}{\mathbb{C}}

\DeclarePairedDelimiter\ceil{\lceil}{\rceil}
\DeclarePairedDelimiter\floor{\lfloor}{\rfloor}

\DeclarePairedDelimiter\abs{|}{|}
\DeclarePairedDelimiter\parens{(}{)}
\DeclarePairedDelimiter\set{\{}{\}}

\DeclarePairedDelimiterX\setof[2]{\{}{\}}{#1\,:\,#2}

\DeclareDocumentCommand{\shad}{O{q}}{\partial_{#1}}
\DeclareDocumentCommand{\upshad}{O{t}}{U^{#1}}

\DeclareDocumentCommand{\k}{O{t}}{k^{#1}}
\DeclareDocumentCommand{\K}{O{t}}{\mathcal{K}^{#1}}
\DeclareDocumentCommand{\S}{O{t}}{\mathcal{S}^{#1}}
\DeclareDocumentCommand{\layer}{O{s}O{\N}}{\binom{#2}{#1}}
\DeclareDocumentCommand{\nlayer}{O{n} m}{\binom{[{#1}]}{#2}}
\DeclareDocumentCommand{\C}{O{s}}{\mathcal{C}_{#1}}

\DeclareDocumentCommand{\iseq}{m m}{n_{#1},n_{{#1}-1},\dots,n_{{#1}-{#2}+1}}
\DeclareDocumentCommand{\is}{O{s}O{s}O{\ell}}{[\iseq{#2}{#3}]_{#1}}
\DeclareDocumentCommand{\seq}{O{s}O{\ell}}{(\iseq{#1}{#2})}

\newcommand{\st}{\ \text{s.t.}\,}
\newcommand{\divides}{\mathrel{\bigm|}}

\newtheorem{thm}{Theorem}[section]
\newtheorem*{thm*}{Theorem}

\newtheorem{cor}[thm]{Corollary}
\newtheorem{conj}[thm]{Conjecture}

\theoremstyle{definition}
\newtheorem{defn}[thm]{Definition}

\newtheorem{obs}[thm]{Observation}
\newtheorem{ex}[thm]{Example}
\newtheorem{case}{Case}
\numberwithin{case}{thm}

\numberwithin{subcase}{case}

\newtheorem{question}[thm]{Question}

\newcommand{\cH}{\mathcal{H}}
\newcommand{\cN}{\mathcal{N}}

\newcommand{\U}{\mathcal{U}}

\newcommand{\scount}[2]{\cN(#1,#2)}
\newcommand{\vsub}[2]{#1^{\downarrow#2}}

\DeclareMathOperator{\cd}{cd}

\DeclareMathOperator{\Dom}{Dom}
\DeclareMathOperator{\dom}{dom}
\renewcommand{\ex}{\mathrm{ex}}
\DeclareMathOperator{\mex}{\mathrm{mex}}

\begin{document}
\pagestyle{plain}

\title{A localized approach to generalized Tur\'an problems}
\author{
Rachel Kirsch\\
George Mason University\\
\texttt{rkirsch4@gmu.edu}
\and
JD Nir\\
Toronto Metropolitan University\\
\texttt{jd.nir@torontomu.ca}}
\date{April 10, 2023}

\maketitle

\begin{abstract}
Generalized Tur\'an problems ask for the maximum number of copies of a graph $H$ in an $n$-vertex, $F$-free graph, denoted by $\ex(n,H,F)$. We show how to extend the new, localized approach of Brada\v{c}, Malec, and Tompkins to generalized Tur\'{a}n problems. We weight the copies of $H$ (typically taking $H=K_t$), instead of the edges, based on the size of the largest clique, path, or star containing the vertices of the copy of $H$, and in each case prove a tight upper bound on the sum of the weights. A consequence of our new localized theorems is an asymptotic determination of $\ex(n,H,K_{1,r})$ for every $H$ having at least one dominating vertex and $\mex(m,H,K_{1,r})$ for every $H$ having at least two dominating vertices.
\end{abstract}

\section{Introduction}

Extremal graph theory is often considered the study of how easily measured global graph parameters, such as the numbers of vertices and edges in a graph, influence its local substructures~\cite{rD17}. An archetypical result due to Tur\'an describes which size cliques a graph is guaranteed to contain based on its order and size:
\begin{thm*}[Tur\'an~\cite{pT41}]
Let $n, r \ge 1$ be integers. If $G$ is a $K_{r+1}$-free graph on $n$ vertices (that is, no subgraph of $G$ is isomorphic to the complete graph $K_{r+1}$), then $G$ contains at most $\frac{n^2}{2}(1-\frac{1}{r})$ edges. This is denoted
\[ \ex(n,K_{r+1}) \le \frac{n^2}{2}\parens[\Big]{1-\frac{1}{r}}. \]

Furthermore, the $K_{r+1}$-free graph on $n$ vertices with the greatest number of edges is the Tur\'an graph, $T_r(n)$, in which the vertices of the graph are partitioned into $r$ parts of sizes as close to equal as possible, and vertices are adjacent if and only if they are in different parts.
\end{thm*}

Tur\'an's theorem has been generalized by many authors. In 2015, Alon and Shikhelman~\cite{AS16}, expanding on several sporadic results (e.g.~\cite{BG08, GPS91, HHKNR13}), introduced the generalized Tur\'an number $\ex(n,H,F)$, which denotes the greatest number of subgraphs of an $F$-free graph on $n$ vertices that are isomorphic to $H$. When maximizing the number of copies of $H\ne K_2$, it is possible to fix the number of edges instead of the number of vertices. Radcliffe and Uzzell~\cite{RU18} introduced the generalized edge Tur\'an number, $\mex(m,H,F)$, which denotes the greatest number of subgraphs that are isomorphic to $H$ in an $F$-free graph on $m$ edges. 

The quantities $\ex(n,H,F)$ and $\mex(m,H,F)$ have motivated many interesting results; see~\cite{CC20,CC21, dG22, GP22, GS18, MNNRW22} for an (incomplete) sample. However, these problems often stretch the notion of ``easily measurable'' properties on which extremal graph theory is based. Though it is not known to be NP-complete, counting the number of subgraphs of $G$ that are isomorphic to $H$ is considered a challenging computation problem. Therefore, practically, it may be difficult to determine whether a specific $G$ contains a forbidden $F$ even when the exact value of $\ex(n,H,F)$ is known.

Recently, Brada\v{c}~\cite{dB22}, based on a conjecture of Balogh and Lidick\'{y}, gave a fundamentally different generalization of Tur\'an's theorem:

\begin{thm*}[Brada\v{c}~\cite{dB22}]
Let $G$ be a graph on $n$ vertices. For each edge $e \in E(G)$, define its weight $w(e)$ as
\[ w(e) = \frac{k}{2(k-1)} \]
where $k$ is the size of the largest clique in $G$ containing $e$. Then
\[ \sum_{e \in E(G)} w(e) \le n^2/4. \]
\end{thm*}

\noindent This theorem has since been applied to Ramsey-Tur\'{a}n problems in \cite{BBL23, BCMM22}.

Brada\v{c}'s result generalizes Tur\'an's theorem in the following sense: if it is known that $G$ is $K_{r+1}$-free, then, noting $w(e)$ is decreasing in $k$, we see $w(e) \ge r/(2(r-1))$ for every edge $e$. Thus
\[ |E(G)| \cdot \frac{r}{(2(r-1))} \le \sum_{e \in E(G)} w(e) \le \frac{n^2}{4} \implies |E(G)| \le \frac{n^2}{2}\parens[\Big]{1-\frac{1}{r}}. \]

The novelty of Brada\v{c}'s result is in the local nature of the weight function. Rather than counting the total number of edges in the entire graph, a global property, the weight we assign to each edge depends only on the neighborhoods of the vertices of that edge which may be computed more efficiently.

Inspired by Brada\v{c}'s result, Malec and Tompkins~\cite{MT22} investigated other results which could be ``localized'' in a similar fashion. In addition to giving a new proof of Brada\v{c}'s result, they proved a local version of another celebrated extremal graph theory result of Erd\H{o}s and Gallai, as well as the LYMB inequality (a generalization of Sperner's theorem on boolean lattices), a generalization of the Erd\H{o}s-Ko-Rado theorem, and a theorem of Erd\H{o}s and Szekeres on sequences.

Both Tur\'an's theorem and the theorem of Erd\H{o}s and Gallai considered by Malec and Tompkins tell us something about the graph based on the number of edges it contains. Keeping in mind that an edge is a clique containing two vertices, natural generalizations of both results to cliques of larger size have been investigated.
In this article, we show that these results, too, admit localized generalizations. In fact, we prove several local results by weighting the cliques or other subgraphs of $G$ of a given size based on the maximum size of various substructures that contain them. In some cases these ``generalizations'' actually establish new extremal results.

Our results follow a general framework. Each theorem concerns a target subgraph $H$, which in many cases is a clique of some fixed size, and a family $\mathcal{F} = \{F_1, F_2, \ldots\}$ of graphs in which $F_i \subseteq F_{i+1}$. We first establish a size function which, given a copy of $H$ in $G$, returns the largest $F_i$ such that some subgraph of $G$ isomorphic to $F_i$ contains $H$ in some meaningful way. Then we define a weight function which depends only on the size function of $H$ and prove a bound on the sum of the weights of every copy of $H$ in $G$. In each case we show that the weight function is a decreasing function of the size function, so a global upper bound on the size function implies a lower bound on the sum of the weights, and we recover a ``non-localized'' theorem.

The rest of this paper is arranged as follows. We begin with some notational conventions and preliminary results in \cref{sec:preliminaries}. Then in \cref{sec:zykov}, we weight $t$-cliques in $G$ by the size of the largest clique containing them to generalize Zykov's theorem, itself a direct extension of Tur\'an's theorem. We weight $t$-cliques by the longest path containing their vertices in \cref{sec:paths}, considering graphs of fixed order in \cref{subsec:pathorder} and graphs of fixed size in \cref{subsec:pathsize}. We weight a broad class of graphs, including cliques, by the size of the largest star containing their vertices in \cref{subsec:graphstars}, in which we prove a family of novel generalized Tur\'an and edge Tur\'an results. We also give a hypergraph version of one of these localized results in \cref{subsec:hypergraphs}. We conclude with some open questions in \cref{sec:conclusion}. 

\section{Preliminaries} \label{sec:preliminaries}

\subsection{Notation}
In addition to standard graph theoretic notation (see~\cite{bB12}, for example), we establish the following conventions. We define the path graph $P_n$ to have $n$ vertices and $n-1$ edges and the star graph $S_r$ to have $r$ leaves (and thus $r+1$ total vertices). Given graphs $G$ and $H$, we let $\scount{H}{G}$ denote the number of subgraphs of $G$ that are isomorphic to $H$. If $\scount{H}{G} = 0$, we say $G$ is $H$-free. As mentioned in the introduction, if $H$ and $F$ are graphs, then
\[ \ex(n,H,F) = \max\set{\scount{H}{G} : |V(G)| = n \text{ and } \text{$G$ is $F$-free}}\]
and
\[ \mex(m,H,F) = \max\set{\scount{H}{G} : |E(G)| = m \text{ and } \text{$G$ is $F$-free}}. \]

We also have need to refer to the copies of $H$ in $G$. When we do so for cliques, we refer to the sets of vertices that span a complete subgraph and write
\[ \K(G) = \set{S \subseteq V(G) : G[S] \cong K_t}.  \]
For a more general graph $H$, we write $\cH(G)$ to refer to the set of (not necessarily induced) subgraphs of $G$ that are isomorphic to $H$.

The size function and weight function(s) in each section are denoted by the notation shown below.

\begin{center}\begin{tabular}{lcccc}
 & Cliques & Paths & Stars (graphs) & Stars (hypergraphs)\\
Section & \ref{sec:zykov} & \ref{sec:paths} &  \ref{subsec:graphstars} & \ref{subsec:hypergraphs}\\
Size function & $\alpha_G$ & $\beta_G$ & $\theta_G$ & $x$\\
Weight function(s) &  $w_G$ & $p_G$, $p'_G$ & $s_G^u$ & $s$
\end{tabular}\end{center}

\subsection{Generalized binomial coefficients}
We use generalized binomial coefficients when working with paths in \cref{sec:paths} and hypergraphs in \cref{sec:stars}. For a real number $x \ge k-1$ and a natural number $k$, the \emph{generalized binomial coefficient} $\binom{x}{k}$ is defined as $(x)(x-1)\cdots(x-k+1)/k!$. When $x < k-1$, we set $\binom{x}{k} = 0$. The function $\binom{x}{k}$ is weakly increasing for all real numbers $x$ and strictly increasing on $x \ge k-1$.

\begin{obs}\label{obs:falling_fact} 
For all $x \in \R$ and $n \in \N$, we have 
\[ 2^n \binom{x}{n} \le \binom{2x}{n}, \]
with strict inequality when $2x+1 > n \ge 2$.
\end{obs}

\begin{proof} The right side is always non-negative and is positive for $2x > n-1$. When $x \le n-1$, the left side is zero. When $x > n-1$, the inequality is equivalent to
\[(2x)(2x-2)\cdots(2x-2n+2) \le (2x)(2x-1)\cdots(2x-n+1),\]
so is strict when $n \ge 2$.
\end{proof}

In \cref{sec:paths} we also use the following observation and theorem.

\begin{obs}
\label{lem:LKKgraph}
	Let $G$ be a graph having at least one edge. Write $\abs{E(G)}$ in the form $\binom{x}{2}$, where $x \ge 2$ is real. Then $\abs{V(G)} \ge x$.
\end{obs}

\begin{proof}
    If $\abs{V(G)} < x$ then $\abs{V(G)} \le \ceil{x}-1$, so $\abs{E(G)} \le \binom{\ceil{x}-1}{2} < \binom{x}{2}$. Here we used that the generalized binomial coefficient $\binom{y}{2}$ is strictly increasing for all $y \ge 1$.
\end{proof}

\begin{thm}[Lov\'{a}sz \cite{L79}]\label{thm:LKKgraph}
	Let $t \ge 2$. Let $G$ be a graph. Write the number of edges of $G$ in the form $\binom{x}{2}$, where $x \ge 1$ is real. Then $\scount{K_t}{G} \le \binom{x}{t}$.
\end{thm}

We also use generalized binomial coefficients in \cref{subsec:hypergraphs}, where we introduce hypergraph definitions and notation before stating \cref{thm:LKK}, a bound on the number of $t$-cliques in a $q$-uniform hypergraph based on the number of edges, which is a version of Lov\'{a}sz's approximate form of the Kruskal-Katona theorem. \cref{thm:LKKgraph} is the special case of \cref{thm:LKK} corresponding to graphs.

\section{Weighting by Maximum Clique Size} \label{sec:zykov}

In this section we prove the following theorem which extends the localized version of Tur\'an's theorem in~\cite{MT22} by assigning weights to cliques of any size, rather than just edges. Then we show that it simultaneously extends a different extension of Tur\'an's theorem due to Zykov.

\begin{thm}\label{thm:local_zykov}
Let $t \ge 2$. For each $T \in \K(G)$, define 
\[ \alpha_G(T) = \max\set{k : T \subseteq V(S) \textnormal{ for some } S \subseteq G \st S \cong K_k} \quad \text{and} \quad w_G(T) = \frac{\alpha_G(T)^t}{\binom{\alpha_G(T)}{t}}. \]
Then $w_G(T)$ is well-defined and decreasing in $\alpha_G(T)$, and
\[ w(G) = \sum_{T \in \K(G)}w_G(T) \le n^t, \]
with equality if and only if $G$ is a balanced multipartite graph with at least $t$ parts.
\end{thm}

Note that by setting $t=2$, we recover the result which is Theorem 1 of both \cite{dB22} and \cite{MT22}. 
\[ \sum_{T\in \K[2](G)} \frac{\alpha_G(e)^2}{\binom{\alpha_G(e)}{2}} = \sum_{e \in E(G)} \frac{2\alpha_G(e)}{\alpha_G(e)-1} \le n^2 \implies \sum_{e \in E(G)} \frac{\alpha_G(e)}{\alpha_G(e)-1} \le \frac{n^2}{2}. \]

\begin{proof}
First, $\alpha_G(T) \ge t$ because $G[T] \cong K_t$, and therefore $w_G(T)$ is well-defined. We observe $w_G(T)$ is decreasing in $\alpha_G(T)$ as we can write
\[ w_G(T) = t!\cdot\frac{\alpha_G(T)^t}{\prod_{i=0}^{t-1} (\alpha_G(T)-i)} = t!\cdot  \prod_{i=1}^{t-1} \parens[\Big]{1+\frac{i}{\alpha_G(T)-i}}, \]
which is a non-empty product of functions decreasing in $\alpha_G(T)$.

Let $G$ be a graph on $n \ge 3$ vertices. Note that if $G$ contains an edge $e$ that is not contained in any $t$-clique, then $e$ also is not contained in any larger clique, so $\K(G-e) = \K(G)$, $w_{G-e}(T) = w_G(T)$ for every $T \in \K(G)$, and $w(G) = w(G-e)$. Therefore we may assume that every edge in $G$ is contained in a $t$-clique. First assume there is $r \ge t$ such that $G$ is complete $r$-partite with (non-empty) parts $A_1, \ldots, A_r$. Then for each $T \in \K(G)$ we have $\alpha_G(T) = r$ and $w_G(T) = r^t/\binom{r}{t}$. This gives
\begin{align*}
w(G) &= \frac{r^t}{\binom{r}{t}}\scount{K_t}{G}\\
	&= \frac{r^t}{\binom{r}{t}}\sum_{S\in\binom{[r]}{t}}\prod_{s\in S}\abs{A_s}
\end{align*}
To bound this sum, we relax the condition that the parts have integral sizes. By symmetry, the real-valued polynomial function
\[ f(x_1,\ldots,x_r) = \sum_{S \in \binom{[r]}{t}} \prod_{i \in S} x_{i}\]
has a unique maximum when each $x_i = n/r$, in which case each of the $\binom{r}{t}$ terms is $(n/r)^t$, and thus
\[ w(G) \le \frac{r^t}{\binom{r}{t}} \cdot \binom{r}{t} \cdot \left(\frac{n}{r}\right)^t = n^t \]
with equality if and only if $|A_i| = n/r$.

Thus we may assume $G$ is not complete multipartite. 
We use a technique introduced by Zykov~\cite{aZ49} sometimes called Zykov symmetrization. Suppose there are $x, y, z \in V(G)$ such that $x \sim z$ but $x \not\sim y \not\sim z$. We show that as long as such vertices exist, we can find $G'$ on $n$ vertices such that $w(G') > w(G)$. If no such vertices exist, then $x \not\sim y$ and $y \not\sim z$ implies $x \not\sim z$, which is to say nonadjacent vertices can be partitioned into equivalence classes and therefore $G$ is complete multipartite. Thus producing such a $G'$ reduces this case to the complete multipartite case, and furthermore proves any graph that meets the bound in the theorem must be complete multipartite.

For convenience, for $v \in V(G)$, define
\[ w_G(v) = \sum_{\substack{T \in \K(G) \\ \text{s.t.} \ v \in T}} w_G(T). \]
Assume without loss of generality that $w_G(x) \ge w_G(z)$ and consider two cases.

\begin{case}$w_G(x) > w_G(y)$.
\end{case}

Introduce a new vertex $x'$, add edges so that $N(x') = N(x)$, and let $G' = G - y + x'$. Consider a clique $K'$ (of any size) that is present in $G'$ but not in $G$. Such a clique must contain $x'$ and so must be contained in $N_{G'}[x'] \cong N_G[x]$. Therefore any $T \in \K(G)$ contained in $K'$ is also contained in some clique $K \subseteq V(G)$ of the same size. Thus every $T \in \K(G) \cap \K(G')$ has $\alpha_G(T) \ge \alpha_{G'}(T)$. As $w_G(T)$ is decreasing in $\alpha_G(T)$, we have $w_G(T) \le w_{G'}(T)$. Then
	\begin{align*}
		\sum_{T \in \K(G')}w_{G'}(T) &= \sum_{T \in \K(G')\setminus\K(G)}w_{G'}(T) + \sum_{T \in \K(G')\cap\K(G)}w_{G'}(T)\\
		&\ge \sum_{T \in \K(G')\setminus\K(G)}w_{G'}(T) + \sum_{T \in \K(G-y)}w_{G}(T)\\
		&= w_{G'}(x') + \parens[\Big]{\sum_{T \in \K(G)}w_G(T) - w_G(y)}\\
		&= w_{G}(x) + \sum_{T \in \K(G)}w_G(T) - w_G(y)\\
		&> \sum_{T \in \K(G)}w_G(T).
	\end{align*}	

\begin{case}$w_G(x) \le w_G(y)$. \end{case}

In this case we introduce two new vertices, $y'$ and $y''$, add edges such that $N(y'') = N(y') = N(y)$, and define $G'' = G - x - z + y' + y''$. As before, every clique $K'$ that is in $G''$ but not in $G$ must contain $y'$ or $y''$ (but not both as $y' \not\sim y''$). Thus once again $K'$ must be contained in $N_{G''}[y'] \cong N_G[y]$ or $N_{G''}[y''] \cong N_G[y]$, so every $T \in \K(G)\cap \K(G'')$ has $\alpha_G(T) \ge \alpha_{G''}(T)$ and $w_G(T) \le w_{G''}(T)$. Therefore
	\begin{align*}
		\sum_{T \in \K(G'')}w_{G''}(T) &= \sum_{T \in \K(G'')\setminus\K(G)}w_{G''}(T) + \sum_{T \in \K(G'')\cap\K(G)}w_{G''}(T)\\
		&\ge \sum_{T \in \K(G'')\setminus\K(G)}w_{G''}(T) +\sum_{T \in \K(G-x-z)}w_{G}(T)\\
		&= w_{G''}(y') + w_{G''}(y'') + \parens[\Big]{\sum_{T \in \K(G)}w_G(T) - w_G(x) - w_G(z) + \sum_{\substack{T \in \K(G) \\ \text{s.t.} \ x,z \in T}}w_G(T)}\\
        &= 2w_G(y) + \sum_{T \in \K(G)}w_G(T) - w_G(x) - w_G(z) + \sum_{\substack{T \in \K(G) \\ \text{s.t.} \ x,z \in T}}w_G(T)\\
		&\ge \sum_{T \in \K(G)}w_G(T) + \sum_{\substack{T \in \K(G) \\ \text{s.t.} \ x,z \in T}}w_G(T)\\
		&> \sum_{T \in \K(G)}w_G(T),
	\end{align*}
where the last step holds because $x \sim z$, and our initial assumption guarantees every edge is contained in a $t$-clique.
\end{proof}

One can also view Theorem~\ref{thm:local_zykov} as a generalization of a theorem of Zykov~\cite{aZ49} (and independently Erd\H{o}s~\cite{pE62}). In what is now considered the first generalized Tur\'an result, Zykov proved that, among $K_{r+1}$-free graphs, the Tur\'an graph $T_r(n)$ maximizes not only the number of edges but the number of cliques of any size.

\begin{thm}[Zykov~\cite{aZ49}]\label{thm:zykov}
Let $G$ be a $K_{r+1}$-free graph on $n$ vertices. Then for any $t \ge 1$,
\[ \scount{K_t}{G} \le \binom{r}{t}\left(\frac{n}{r}\right)^t. \]
\end{thm}

Note that Zykov actually proved the stronger result that $\ex(n, K_t, K_{r+1}) = \scount{K_t}{T_r(n)}$, though these results agree when $r \divides n$. 

We can prove Theorem~\ref{thm:zykov} as a consequence of Theorem~\ref{thm:local_zykov}:

\begin{proof}[Proof of \cref{thm:zykov}]
Let $G$ be a an $n$-vertex, $K_{r+1}$-free graph. Then for each $T \in \K(G)$ we have $\alpha_G(T) \le r$ and, as $w_G(T)$ is decreasing in $\alpha_G(T)$, $w_G(T) \ge r^t/\binom{r}{t}$. Thus by Theorem~\ref{thm:local_zykov},
\[ \scount{K_t}{G} \cdot \frac{r^t}{\binom{r}{t}} = \sum_{T \in \K(G)} \frac{r^t}{\binom{r}{t}} \le \sum_{T \in \K(G)} w_G(T) \le n^t \]
and so
\[ \scount{K_t}{G} \le \frac{\binom{r}{t}}{r^t} \cdot n^t = \binom{r}{t} \left(\frac{n}{r}\right)^t. \qedhere\]
\end{proof}

\section{Weighting by Maximum Path Length} \label{sec:paths}

In 1959, Erd\H{o}s and Gallai~\cite{EG59} proved that every graph with $n$ vertices and $m$ edges contains a path of length at least $2m/n$, and as a consequence $\ex(n, K_2, P_{r+1}) \le \frac{(r-1)n}{2}$, which, when $r \divides n$, is achieved by a disjoint union of copies of $K_r$. Luo~\cite{rL18} determined $\ex(n,K_t,P_{r+1})$ asymptotically, and Chakraborti and Chen~\cite{CC20} then determined $\mex(m,K_t,P_{r+1})$. Malec and Tompkins~\cite{MT22} gave a local version of the Erd\H{o}s-Gallai result. To extend their result to the $t$-clique versions, considering both graphs of fixed order and graphs of fixed size, we define a more general size function for paths.
\begin{defn}
Let $G$ be a graph on $n$ vertices and $t \ge 2$. For each $T \in \K(G)$, define
\[ \beta_G(T) = \max\set{k : T \subseteq V(S) \textnormal{ for some } S \subseteq G \st S \cong P_{k+1}} . \]
\end{defn}

\subsection{Paths in graphs of fixed order}\label{subsec:pathorder}

In this section we extend a result of Malec and Tompkins~\cite[Theorem 2]{MT22} to cliques, which we will demonstrate also generalizes a result of Luo~\cite{rL18}.

\begin{thm}\label{thm:local_luo_path}
Let $G$ be a graph on $n$ vertices and $t \ge 2$. Define
\[ p_G(T) = \frac{1}{\binom{\beta_G(T)}{t-1}}. \]
Then $p_G(T)$ is well-defined and decreasing in $\beta_G(T)$, and
\[ \sum_{T \in \K(G)}p_G(T) \le \frac{n}{t}, \]
with equality if and only if $G$ is a disjoint union of complete graphs of order at least $t$.
\end{thm}

We note that setting $t=2$ does not quite recover \cite[Theorem 2]{MT22} as our weight function is not equivalent at $t=2$. Malec and Tompkins define $\beta_{MT}(e)$ to be the longest path containing the edge $e$ as a subgraph. This definition does not extend to larger cliques as paths do not contain them as subgraphs. Instead, we merely require that all vertices of the clique appear in the path. Thus when $t=2$, unlike for Malec and Tompkins, the vertices of our edge may occur in the path without the edge being part of the path.

\begin{proof}
First, we note any ordering of the vertices in any $T \in \K(G)$ is a path of length $t-1$, so $\beta_G(T) \ge t-1$ and $\binom{\beta_G(T)}{t-1} > 0$. Therefore $p_G(T)$ is well-defined and decreasing in $\beta_G(T)$.

We proceed by induction on $n$. When $1 \le n \le t-1$ we have
\[ \sum_{T \in \K(G)}p_G(T) = 0 < \frac{n}{t}. \]
	
If $G$ is not connected, let $C_1, \ldots, C_q$ be the components of $G$. Applying the inductive hypothesis to each component, we have
	\begin{align*}
		 \sum_{T \in \K(G)}p_G(T) &= \sum_{i=1}^q \sum_{T \in \K(C_i)} p_G(T)\\
   &= \sum_{i=1}^q \sum_{T \in \K(C_i)} p_{C_i}(T)\\
		 &\le \sum_{i=1}^q \frac{\abs{V(C_i)}}{t}\quad\text{by induction}\\
		 &= \frac{n}{t}.
	\end{align*}
Therefore we may assume $G$ is connected and $n \ge t$. Let $r$ be the length of a longest path $P \cong P_{r+1}$ in $G$.
	
	\begin{case}\label{case:cycle}
		There exists a cycle $C$ containing the vertices of $P$.
	\end{case}

	Suppose for the sake of contradiction that there is a vertex $u$ that is not on the cycle $C$. Since $G$ is connected, there is a path from $u$ to $C$ and then around $C$, which is longer than $P$, contradicting that $P$ is a longest path. Therefore every vertex of $G$ is on $C$. Each $T \in \K(G)$ is contained in a path of length $n-1$ and has $p_G(T) = 1/\binom{n-1}{t-1}$.
	
	Hence,
	\[
		\sum_{T \in \K(G)}p_G(T) = \sum_{T \in \K(G)}\frac{1}{\binom{n-1}{t-1}} = \frac{\scount{K_t}{G}}{\binom{n-1}{t-1}} \le \frac{\binom{n}{t}}{\binom{n-1}{t-1}} = \frac{n}{t},
	\]
	and equality implies $G \cong K_n$, with $n \ge t$.
 
	\begin{case}\label{case:nocycle}
		There does not exist a cycle $C$ containing the vertices of $P$.
	\end{case}
	
 	Let $v$ and $w$ be the endpoints of $P$. Then $\set{v,w} \notin E(G)$.	Label the vertices of $P$ in order as $(v=u_1, u_2, \ldots, u_{r+1}=w)$. Let $V = \set{i : v \sim u_i}$ and $W = \set{i: w \sim u_{i-1}}$. If $i \in V\cap W$ then $(u_i, u_1, u_2, \ldots, u_{i-1}, u_{r+1}, u_r, \ldots, u_i)$ is a cycle containing the vertices of $P$, so $V\cap W = \emptyset$. Since $V, W \subseteq \set{2,\ldots, r+1}$, we have  $d(v) + d(w) = \abs{V}+\abs{W}=\abs{V\cup W} \le r$. Assume, without loss of generality, that $d(v) \le r/2$.

    Let $R(v)$ be the set of $t$-cliques of $G$ that contain $v$. Then $\abs{R(v)} \le \binom{d(v)}{t-1}$. The fact that $P$ is a longest path implies both that $\beta_G(T) \le r$ and that every $t$-clique $T$ in $R(v)$ is contained in $V(P)$ (so $\beta_G(T) \ge r$), and therefore $p_G(T) = 1/\binom{r}{t-1}$ for every $T \in R(v)$. For any $T \in \K(G-v)$, we see $\beta_{G-v}(T) \le \beta_G(T)$ and therefore, as $p_G$ is decreasing in $\beta_G(T)$, $p_{G-v}(T) \ge p_G(T)$.
    
  Applying the inductive hypothesis to $G-v$, we get
 	
 	\begin{align*}
 		\sum_{T \in \K(G)}p_G(T) &= \sum_{T \in \K(G-v)}p_G(T) + \sum_{T \in R(v)}p_G(T)\\
            &\le \sum_{T \in \K(G-v)}p_{G-v}(T) + \sum_{T \in R(v)}p_G(T)\\
 		&\le \frac{n-1}{t} + \frac{\abs{R(v)}}{\binom{r}{t-1}}\\
 		&\le \frac{n-1}{t} + \frac{\binom{d(v)}{t-1}}{\binom{r}{t-1}}\\
 		&\le \frac{n-1}{t} + \frac{\binom{r/2}{t-1}}{\binom{r}{t-1}}\\
            &\le \frac{n-1}{t} +\frac{(1/2)^{t-1} \binom{r}{t-1}}{\binom{r}{t-1}} & \text{by \cref{obs:falling_fact}}\\
 		&= \frac{n-1}{t} + \frac{1}{2^{t-1}}\\
 		&\le \frac{n-1}{t} + \frac{1}{t}\\
 		&= \frac{n}{t}
 	\end{align*}
 	as $2^{t-1} \ge t$ when $t \ge 2$. 
When $t \ge 3$, we see $2^{t-1} > t$ and thus if equality holds, no component is in Case \ref{case:nocycle}, so every component is in Case \ref{case:cycle}, and every component is a clique.

When $t=2$, the claim that equality holds only for cliques follows from~\cite[Theorem 2]{MT22}. Though, as noted, our definition of $\beta_G(T)$ does not quite match their $\beta_{MT}(e)$, we have $\beta_{MT}(e) \le \beta_G(e)$ as any path containing an edge also contains the vertices of that edge. Because
\[ p_G(T) = \frac{1}{\binom{\beta_G(T)}{2-1}} = \frac{1}{\beta_G(T)} \]
is decreasing and matches Malec and Tompkins' weight, we see
\[ \sum_{T \in \K[2](G)} p_G(T) = \sum_{e \in E(G)} \frac{1}{\beta_G(e)} \le \sum_{e \in E(G)} \frac{1}{\beta_{MT}(e)}. \]
Therefore if equality holds in our result, it also holds in Malec and Tompkins' result, and by Theorem~2 in \cite{MT22}, $G$ is a disjoint union of complete graphs.
\end{proof}

In 2018, Luo~\cite{rL18} extended Erd\H{o}s and Gallai's theorem by showing disjoint unions of cliques maximize the number of cliques of any size.

\begin{thm}[Luo~\cite{rL18}]\label{thm:luo_path}
Let $G$ be a $P_{r+1}$-free graph on $n$ vertices. Then for any $t \le r$,
\[ \scount{K_t}{G} \le \frac{n}{r} \binom{r}{t}. \]
Furthermore, equality holds if and only if $r \divides n$ and $G$ is a disjoint union of copies of $K_r$.
\end{thm}

We note that Theorem~\ref{thm:local_luo_path} generalizes Theorem~\ref{thm:luo_path}:

\begin{proof}
	Let $G$ be a $P_{r+1}$-free graph on $n$ vertices. In such a graph $G$, every $T \in \K(G)$ has $\beta_G(T) \le r-1$, so $p_G(T) \ge \frac{1}{\binom{r-1}{t-1}}$. By Theorem~\ref{thm:local_luo_path},
	\[
		\frac{\scount{K_t}{G}}{\binom{r-1}{t-1}} = \sum_{T \in \K(G)} \frac{1}{\binom{r-1}{t-1}} \le \sum_{T \in \K(G)}p_G(T) \le \frac{n}{t},
	\]
	so
	\[
	\scount{K_t}{G} \le \frac{n}{t}\binom{r-1}{t-1} = \frac{n}{r}\binom{r}{t}.
	\]
	By \cref{thm:local_luo_path}, equality implies that every component is a $P_{r+1}$-free clique, and that every $t$-clique is contained in a $P_r$, so $G$ is a disjoint union of copies of $K_r$.
\end{proof}

\subsection{Paths in graphs of fixed size}\label{subsec:pathsize}
Chakraborti and Chen \cite{CC20} asked and answered the edge variant of this question by determining $\mex(m,K_t,P_{r+1})$ exactly for all values of the parameters $m$, $t$, and $r$. We prove a localized form of their result, then show that it implies an asymptotic determination of $\mex(m,K_t,P_{r+1})$. 

\begin{thm}\label{thm:local_mex_path}
Let $t \ge 3$. Let $G$ be a graph having $m$ edges. For each $T \in \K(G)$, define
\[ p'_G(T) = \frac{1}{\binom{\beta_G(T)-1}{t-2}}. \]
Then $p'_G(T)$ is well-defined and decreasing in $\beta_G(T)$, and
\[ \sum_{T \in \K(G)}p'_G(T) \le \frac{m}{\binom{t}{2}}, \]
with equality if and only if $G$ is a disjoint union of complete graphs of order at least $t$ and any number of isolated vertices. 
\end{thm}

\begin{proof}
As before, any ordering of the vertices in any $T \in \K(G)$ is a path of length $t-1$, so $\beta_G(T) \ge t-1$ and $\binom{\beta_G(T)-1}{t-2} > 0$. Therefore $p'_G(T)$ is well-defined and decreasing in $\beta_G(T)$.

We proceed by induction on $m$. When $1 \le m \le \binom{t}{2}-1$ we have
\[ \sum_{T \in \K(G)}p'_G(T) = 0 < \frac{m}{\binom{t}{2}}. \]

If $G$ is not connected, let $C_1, \ldots, C_q$ be the components of $G$. Applying the inductive hypothesis to each component, we have
	\begin{align*}
		 \sum_{T \in \K(G)}p'_G(T) &= \sum_{i=1}^q \sum_{T \in \K(C_i)} p'_G(T)\\
        &= \sum_{i=1}^q \sum_{T \in \K(C_i)} p'_{C_i}(T)\\
		 &\le \sum_{i=1}^q \frac{\abs{E(C_i)}}{\binom{t}{2}}\quad\text{by induction}\\
		 &= \frac{m}{\binom{t}{2}}.
	\end{align*}

Therefore we may assume $G$ is connected and $m \ge \binom{t}{2}$. Let $r$ be the length of a longest path $P \cong P_{r+1}$ in $G$.

\begin{case}\label{case:cycle_edge}
		There exists a cycle $C$ containing the vertices of $P$.
	\end{case}

	Suppose for the sake of contradiction that there is a vertex $u$ that is not on the cycle $C$. Since $G$ is connected, there is a path from $u$ to $C$ and then around $C$, which is longer than $P$, contradicting that $P$ is a longest path. Therefore every vertex of $G$ is on $C$. Each $T \in \K(G)$ is contained in a path of length $\abs{V(G)}-1$, and has $p'_G(T) = 1/\binom{\abs{V(G)}-2}{t-2}$. Let $x \ge t$ be a real number satisfying $m = \binom{x}{2}$. Then we have $\abs{V(G)} \ge x$ by \cref{lem:LKKgraph} and, by \cref{thm:LKKgraph}, $\scount{K_t}{G} \le \binom{x}{t}$. Hence,
	\[
		\sum_{T \in \K(G)}p'_G(T) = \sum_{T \in \K(G)}\frac{1}{\binom{\abs{V(G)}-2}{t-2}} = \frac{\scount{K_t}{G}}{\binom{\abs{V(G)}-2}{t-2}} \le \frac{\binom{x}{t}}{\binom{x-2}{t-2}} = \frac{\binom{x}{2}}{\binom{t}{2}} = \frac{m}{\binom{t}{2}},
	\]
    and equality implies $\binom{\abs{V(G)}-2}{t-2} = \binom{x-2}{t-2}$ with $x-2 \ge (t-2)-1$, so $x = \abs{V(G)}$. Then $\scount{K_t}{G} = \binom{\abs{V(G)}}{t}$, so $G\cong K_x$ (and $x \ge t$).

 \begin{case}\label{case:nocycle_edge}
		There does not exist a cycle $C$ containing the vertices of $P$.
	\end{case}

 	Let $v$ and $w$ be the endpoints of $P$. Then $\set{v,w} \notin E(G)$. Label the vertices of $P$ in order as $(v=u_1, u_2, \ldots, u_{r+1}=w)$. Let $V = \set{i : v \sim u_i}$ and $W = \set{i: w \sim u_{i-1}}$. If $i \in V\cap W$ then $(u_i, u_1, u_2, \ldots, u_{i-1}, u_{r+1}, u_r, \ldots, u_i)$ is a cycle containing the vertices of $P$. Because $V, W \subseteq \set{2,\ldots, r+1}$, we have  $d(v) + d(w) = \abs{V}+\abs{W}=\abs{V\cup W} \le r$. Assume, without loss of generality, that $d(v) \le r/2$.
 	
 	Let $R(v)$ be the set of $t$-cliques of $G$ that contain $v$. Then $\abs{R(v)} \le \binom{d(v)}{t-1}$. Since $P$ is a longest path, $v$ and $w$ have no neighbors outside of $P$, and so every $t$-clique in $R(v)$ is contained in $V(P)$. Every $T \in R(v)$ has $p'_G(T) = 1/\binom{r-1}{t-2}$ because it is contained in $P$ (and there are no longer paths). As noted above, we have $\beta_{G-v}(T) \le \beta_G(T)$ for any $T \in \K(G-v)$ and as $p'_G$ is also decreasing in $\beta_G(T)$, $p'_{G-v}(T) \ge p'_G(T)$. Applying the inductive hypothesis to $G-v$, we get
 	
 	\begin{align*}
 		\sum_{T \in \K(G)}p'_G(T) &= \sum_{T \in \K(G-v)}p'_G(T) + \sum_{T \in R(v)}p'_G(T)\\
        &\le \sum_{T \in \K(G-v)}p'_{G-v}(T) + \sum_{T \in R(v)}p'_G(T)\\
		&\le \frac{m-d(v)}{\binom{t}{2}} + \frac{\abs{R(v)}}{\binom{r-1}{t-2}}\\
		&\le \frac{m-d(v)}{\binom{t}{2}} + \frac{\binom{d(v)}{t-1}}{\binom{r-1}{t-2}}\\
		&= \frac{m-d(v)}{\binom{t}{2}} + \frac{\binom{d(v)-1}{t-2}}{\binom{r-1}{t-2}} \cdot \frac{d(v)}{t-1}\\
 		&\le \frac{m-d(v)}{\binom{t}{2}} + \frac{\binom{r/2-1}{t-2}}{\binom{r-1}{t-2}} \cdot \frac{d(v)}{t-1}\\
    & \le \frac{m-d(v)}{\binom{t}{2}} +\frac{(1/2)^{t-2} \binom{r-2}{t-2}}{\binom{r-1}{t-2}}  \cdot \frac{d(v)}{t-1} & \text{by \cref{obs:falling_fact}}\\
 		&= \frac{m-d(v)}{\binom{t}{2}} + \frac{r-t+1}{r-1} \cdot \frac{1}{2^{t-1}} \cdot \frac{d(v)}{\frac{t-1}{2}}\\
  		&< \frac{m-d(v)}{\binom{t}{2}} + \frac{d(v)}{\binom{t}{2}}\\
  		&= \frac{m}{\binom{t}{2}}
  \end{align*}
 	as $2^{t-1} > t$ and $\frac{r-t+1}{r-1} < 1$ when $t \ge 3$. Thus, if equality holds, no component is in Case \ref{case:nocycle_edge}, so every component is in Case \ref{case:cycle_edge}, and every component is a clique.
\end{proof}

As mentioned, Chakraborti and Chen~\cite{CC20} determined $\mex(m,K_t,P_{r+1})$ exactly, from which one can derive the following weaker result:

\begin{thm}[Chakraborti and Chen~\cite{CC20}] \label{thm:cc_mex_paths}
For any $3 \le t \le r$, if $G$ is a $P_{r+1}$-free graph with $m$ edges, then
\[ \scount{K_t}{G} \le \frac{m}{\binom{r}{2}} \cdot \binom{r}{t}.\]
Furthermore, equality holds if and only if $\binom{r}{2}$ divides $m$, and $G$ is a disjoint union of copies of $K_r$.
\end{thm}

We use Theorem~\ref{thm:local_mex_path} to prove Theorem~\ref{thm:cc_mex_paths}:

\begin{proof}
Let $G$ be a $P_{r+1}$-free graph. Then for each $T \in \K(G)$, we have $\beta_G(T) \le r-1$ and as $p'_G(T)$ is decreasing in $\beta_G$,
\[ \scount{K_t}{G} \cdot \frac{1}{\binom{r-2}{t-2}} \le \sum_{T \in \K(G)} p'_G(T) \le \frac{m}{\binom{t}{2}}\]
and therefore
\[ \scount{K_t}{G} \le \frac{m}{\binom{t}{2}} \cdot \binom{r-2}{t-2} = \frac{m}{\binom{r}{2}} \cdot \binom{r}{t}.\]
Equality implies $G$ is a disjoint union of $P_{r+1}$-free complete graphs, and every $t$-clique is contained in a $P_r$, so $G$ is a disjoint union of copies of $K_r$.
\end{proof}

\section{Weighting by Maximum Star Size} \label{sec:stars}

\subsection{Graphs} \label{subsec:graphstars}

In this section we consider generalized extremal problems of the form $\ex(n,H,S_r)$, forbidding the star with $r$ leaves. Unlike in previous sections where $H$ was a clique, here we consider a broader range of graphs $H$. We use $\cH(G)$ to denote the set of (not necessarily induced) subgraphs of $G$ isomorphic to $H$. 

Given a graph $G$ and a non-empty set of vertices $U \subseteq V(G)$, the \emph{common neighborhood} of $U$ in $G$ is the set of vertices of $G$ adjacent to each vertex in $U$, or equivalently the intersection of the neighbor sets of each vertex of $U$. The common degree of $U$, which we denote by $\cd_G(U)$, is the size of the common neighborhood. Note that $U$ is disjoint from its common neighborhood.

Denote the collection of dominating vertices of a graph $G$ by $\Dom(G)$, and let $\dom(G) = |\Dom(G)|$. Note that for any $U \subseteq \Dom(G)$, $U$ is a clique, and the common neighborhood of $U$ is $V(G) \setminus U$. Additionally, if $U, U' \subseteq \Dom(G)$ and $|U| = |U'|$, then $G-U \cong G-U'$. For $u \le \dom(G)$, we write $\vsub{G}{u}$ to denote the graph that results from removing $u$ dominating vertices from $G$. 

The main result of this section, \cref{thm:local_star}, provides a template for localized bounds on the number of copies of $H$ in a graph based on the number of sets of dominating vertices of a given size contained in $H$. We will focus on the cases where these sets have size one or two, but we provide the theorem in full generality.

\begin{thm}\label{thm:local_star}
Let $H$ be a graph on $t$ vertices with $\dom(H) \ge 1$. 
For every graph $G$ and each $T \in \cH(G)$, define
\[ \theta_G(T) = \max\set{k : \exists v \in V(T) \subseteq V(S) \st V(S) \subseteq N[v] \text{ and } S_k \cong S \subseteq G} = \max\set{d_G(v) : v \in \Dom(T)}, \]
and for each $1 \le u \le \dom(H)$, define
\[ s^u_G(T) = \frac{1}{\binom{\theta_G(T)-u+1}{t-u}}. \]
Then for any $1 \le u \le \dom(H)$, $s_G^u(T)$ is well-defined and decreasing in $\theta_G(T)$, and
\[\sum_{T \in \cH(G)} s^u_G(T) \le \frac{\scount{\vsub{H}{u}}{K_{t-u}}}{\binom{\dom(H)}{u}}\cdot \scount{K_{u}}{G}. \]
Furthermore, equality holds
\begin{itemize}
   \item if and only if $G$ has minimum degree at least $t-1$, if $H=S_{t-1}$ for some $t \ge 3$,
   \item if and only if $G$ contains no isolated vertices and every component of $G$ is regular, if $u=1$ and $H=S_1$, and
   \item if and only if every component of $G$ that contains a $u$-clique is a complete graph on at least $t$ vertices, if $u \ge 2$ or $H$ is not a star.
\end{itemize}
\end{thm}

\begin{proof}
First, in any $T \in \cH(G)$, there is a dominating vertex $v$ of $T$ which is the center of a star with (at least) $t-1$ leaves, so $\theta_G(T) \ge t-1$ and $\binom{\theta_G(T)-u+1}{t-u} > 0$. Therefore $s^u_G(T)$ is well-defined and decreasing on $\theta_G(T)$.

Let $T \in \cH(G)$ and $U \subseteq \Dom(T)$ such that $\abs{U} = u$. Recall $U$ is a clique. Thus for any $v \in U$, the vertices in the common neighborhood of $U$ in $G$, together with the vertices $U \setminus \set{v}$, are each adjacent to $v$, forming a copy of $S_{\cd_G(U)+u-1}$ whose center is $v \in V(T)$ and which contains all vertices of $T$ (because $v \in U \subseteq \Dom(T)$).  Therefore we have $\theta_G(T) \ge \cd_G(U) + u-1$, or $\theta_G(T) - u + 1 \ge \cd_G(U)$.

For each $T \in \cH(G)$, there are $\binom{\dom(H)}{u}$ sets $U \in \K[u](G)$ such that each vertex of $U$ is dominating in $T$. For each $U \in \K[u](G)$, the number of copies $T$ of $H$ in $G$ for which $U \subseteq \Dom(T)$ is at most $\binom{\cd_G(U)}{t-u}\scount{\vsub{H}{u}}{K_{t-u}}$: we choose $t-u$ vertices from the common neighborhood of $U$ in $G$ to fill out $T$ and then choose how to embed the vertices of $\vsub{H}{u}$, which can be done in at most as many ways as embedding them into a clique of the same size. Therefore
	\begin{align}
	\binom{\dom(H)}{u}\sum_{T \in \cH(G)} s^u_G(T) &= \sum_{U \in \K[u](G)} \sum_{\substack{T\in \cH(G)\\U \subseteq \Dom(T)}} s^u_G(T) \notag\\
	&= \sum_{U \in \K[u](G)} \sum_{\substack{T\in \cH(G)\\U \subseteq \Dom(T)}} \frac{1}{\binom{\theta_G(T)-u+1}{t-u}} \notag\\
	&\le \sum_{U \in \K[u](G)} \sum_{\substack{T\in \cH(G)\\U \subseteq \Dom(T)}} \frac{1}{\binom{\cd_G(U)}{t-u}} \label{eq:theta}\\ 
	&\le \sum_{U \in \K[u](G)} \scount{\vsub{H}{u}}{K_{t-u}}\frac{\binom{\cd_G(U)}{t-u}}{\binom{\cd_G(U)}{t-u}} \label{eq:count}\\
	&= \scount{\vsub{H}{u}}{K_{t-u}}\scount{K_u}{G} \notag.
	\end{align}
which establishes the inequality for all graphs $G$.

We now consider when equality holds in the bound. For ease of discussion, if $U$ is contained in a copy $T$ of $H$ such that $U \subseteq \Dom(T)$, we say that $T$ is an extension of $U$. Equality in the bound requires equality in equations \eqref{eq:theta} and \eqref{eq:count}. In order for equality to hold in \eqref{eq:theta}, every pair $(U,T)$ where $T$ is an extension of $U$ must satisfy $\cd_G(U) = \theta_G(T)-u+1$. Equality in \eqref{eq:count} requires two conditions: first, for every $U \in \K[u](G)$, $G$ must contain an extension of $U$ as otherwise this choice of $U$ contributes
\[ \sum_{\substack{T\in \cH(G)\\U \subseteq \Dom(T)}} \frac{1}{\binom{\cd_G(U)}{t-u}} = 0\]
instead of
\[ \sum_{\substack{T\in \cH(G)\\U \subseteq \Dom(T)}} \frac{1}{\binom{\cd_G(U)}{t-u}} = \scount{\vsub{H}{u}}{K_{t-u}} > 0. \]
In particular, if equality holds in \eqref{eq:count}, $\cd_G(U) \ge t-u$ for every $U \in \K[u](G)$. Additionally, every $(t-u)$-set in the common neighborhood of $U$ must contain $\scount{\vsub{H}{u}}{K_{t-u}}$ copies of $\vsub{H}{u}$, or equivalently, the common neighborhood of $U$ must contain the same number of copies of $\vsub{H}{u}$ as a clique of the same size.

For each of the three cases delineated in the theorem, we consider the conditions under which equality holds.

\begin{case}\label{case:star_not_edge}
Suppose $H = S_{t-1}$ for some $t \ge 3$.
\end{case}

In this case we have $\dom(H) = 1$ and thus $u=1$. This means that for any $U \in \K[u](G)$, there is $v \in V(G)$ such that $U = \{v\}$, and we have $\cd_G(U) = d_G(v)$. For any copy $T$ of $H$ in $G$, there is a unique $U \subseteq \Dom(T)$ of size $u=1$, and if $U = \{v\}$, then $\theta_G(T) = d_G(v) = \cd_G(U)$. We conclude that equality holds in \eqref{eq:theta} for every graph $G$. If $w$ is the center of $H=S_{t-1}$, then $H-w = \vsub{H}{1}$ is an independent set, so $\scount{\vsub{H}{1}}{K_{t-1}} = 1$. For any graph $G$ and any $v \in V(G)$, if $d_G(v) \ge t-1$, then the number of copies $T$ of $H$ with $v$ at the center is exactly $\binom{d_G(v)}{t-1} = \binom{\cd_G(U)}{t-1}$ and equality holds in \eqref{eq:count}. If $G$ has minimum degree at least $t-1$ then equality holds for every vertex, so equality holds in the bound. On the other hand, if $d_G(v) < t-1$ for any $v \in V(G)$, then there are no extensions of $\{v\}$ and equality does not hold in \eqref{eq:count}. As equality in the bound requires equality for each $U \in \K[u](G)$, we conclude equality holds in this case if and only if $d_G(v) \ge t-1$ for every $v \in V(G)$, which is to say $G$ has minimum degree at least $t-1$.

\begin{case}
Suppose $u=1$ and $H=S_1$.
\end{case}

This case is similar to \cref{case:star_not_edge}: we have $u=1$, so the set of $u$-cliques of $G$ corresponds to the vertices of $G$, $\vsub{H}{1}$ is an independent set so $\scount{\vsub{H}{1}}{K_{t-1}} = 1$, and equality holds in \eqref{eq:count} if and only if each vertex has degree at least $t-1 = 2-1$, i.e. $G$ contains no isolated vertices. The difference between the cases is that $\dom(H) = 2$ when $H=S_1$, and therefore rather than being determined by the degree of a single vertex,
\[ \theta_G(T=\{u,v\}) = \max\{d_G(u), d_G(v)\}. \]
This introduces a circumstance in which equality may not hold in \eqref{eq:theta}: if $T = \{u,v\}$ and $d_G(u) > d_G(v)$, then for $U = \{v\}$,
\[ \frac{1}{\binom{\theta_G(T)}{t-u}} = \frac{1}{d_G(u)} < \frac{1}{d_G(v)} =  \frac{1}{\binom{\cd_G(U)}{t-u}}. \]
Thus equality in \eqref{eq:theta} holds if and only if $d_G(u) = d_G(v)$ for every $T = \{u,v\}$, which is to say each component of $G$ is regular. We conclude a graph $G$ meets the bound if and only if $G$ contains no isolated vertices and each component of $G$ is regular.

\begin{case}
Suppose $u \ge 2$ or that $H$ is not a star.
\end{case}

First we prove the given conditions are sufficient to achieve the bound. When $G$ is a disjoint union of complete graphs on at least $t$ vertices and any number of components without $u$-cliques, let $K_r$, $r \ge t \ge u$, be a component of $G$ that contains a $u$-clique. Then $\theta_G(T) = r-1$ for every $T \in \cH(K_r)$, and $s_G^u(T) = 1/\binom{r-u}{t-u}$. The number of copies of $H$ in this component is $\binom{r}{u}\binom{r-u}{t-u}\scount{\vsub{H}{u}}{K_{t-u}}/\binom{\dom{H}}{u}$, which can be seen by first choosing a $u$-clique of the $K_r$ to act as a selected set of $u$ dominating vertices of $H$, then choosing $t-u$ of the $r-u$ other vertices in the same component, then choosing an embedding of $\vsub{H}{u}$ into those vertices. Each copy of $H$ is counted this way once for each choice of selected set of $u$ dominating vertices of $H$. Therefore
\[
\sum_{T \in \cH(K_r)}s_G^u(T) = \frac{1}{\binom{r-u}{t-u}}\binom{r}{u}\binom{r-u}{t-u}\scount{\vsub{H}{u}}{K_{t-u}}/\binom{\dom(H)}{u} = \frac{\scount{\vsub{H}{u}}{K_{t-u}}}{\binom{\dom(H)}{u}}\scount{K_u}{K_r}.
\]
Note that if a component $C$ of $G$ contains no $u$-cliques, it also contains no copies of $T$, so we have
\[ \sum_{T \in \cH(C)} s^u_C(T) = 0 = \scount{K_u}{C}, \]
and thus $C$ does not contribute to either side of the bound. If $C_1, \ldots, C_k$ are the components of $G$ containing $u$-cliques, then
\[
\sum_{T \in \cH(G)}s_G^u(T) = \sum_{i=1}^k \sum_{T \in \cH(C_i)}s_{C_i}^u(T) = \frac{\scount{\vsub{H}{u}}{K_{t-u}}}{\binom{\dom(H)}{u}}\sum_{i=1}^k \scount{K_u}{C_i} = \frac{\scount{\vsub{H}{u}}{K_{t-u}}}{\binom{\dom(H)}{u}}\scount{K_u}{G}.
\]
We conclude that any graph $G$ in which every component that contains a $u$-clique is a complete graph on at least $t$ vertices meets the bound.

Now if $G$ is a graph meeting the bound, we prove every component of $G$ containing a $u$-clique must be a clique containing at least $t$ vertices. We begin with three claims.

\medskip

\noindent\textbf{Claim 1}: If $G$ is a graph meeting the bound and $U \in \K[u](G)$, then each vertex in $V(G) \setminus U$ that is adjacent to some $v \in U$ must be in the common neighborhood of $U$.

\medskip

Suppose otherwise: there is $v \in U$ and $x$ not in the common neighborhood of $U$ such that $x \sim v$. As $G$ meets the bound, equality holds in \eqref{eq:count}, so we may assume there exists an extension $T$ of $U$. Note that $T \subseteq N[v]$. As $v$ is adjacent the common neighborhood of $U$, each other vertex of $U$, and $x$, we see 
\[ \theta_G(T) = \max\{d_G(w) : w \in \Dom(T)\} \ge d_G(v) \ge \cd_G(U) +u -1 + 1 > \cd_G(U)+u - 1 \]
and equality does not hold in \eqref{eq:theta}, contradicting that $G$ met the bound.

\medskip

\noindent\textbf{Claim 2}: If $u \ge 2$, $G$ is a graph meeting the bound, and $U \in \K[u](G)$, then $U$ and its common neighborhood form a clique in $G$.

\medskip

Suppose $v,w \in U$ and that $x$ and $y$ are vertices in the common neighborhood of $U$ such that $x \not\sim y$. Define $U' = U-v+x$. If there are no extensions $T$ of $U'$, then equality will not hold in \eqref{eq:count}, so assume such a $T$ exists. Then $y$ is not in the common neighborhood of $U'$, as $y \not\sim x$, but $y \sim w \in U \cap U'$, so by the contrapositive of Claim 1, equality does not hold.

\medskip

\noindent\textbf{Claim 3}: If $\vsub{H}{u}$ contains at least one edge, $G$ is a graph meeting the bound, and $U \in \K[u](G)$, then $U$ and its common neighborhood form a clique in $G$.

\medskip

We prove the contrapositive. Assume that some $U \in \K[u](G)$ and its common neighborhood do not form a clique in $G$. Let $a$ and $b$ be non-adjacent vertices in the common neighborhood of $U$ and let $\set{x,y}$ be an edge in $\vsub{H}{u}$. If $G$ contains no extensions $T$ of $U$, then equality does not hold in \eqref{eq:count}, so we may assume such an extension does exist, and therefore $\cd_G(U) \ge t-u = \abs{V(\vsub{H}{u})}$ as the common neighborhood of $U$ contains at least $V(T) \setminus U$. Therefore there exists an injective function $f$ from $V(\vsub{H}{u})$ to the common neighborhood of $U$ such that $f(x) = a$ and $f(y)=b$. Note that $f$ is not a homomorphism, demonstrating that not every injective function from $\vsub{H}{u}$ to the common neighborhood of $U$ is an injective homomorphism. By contrast, every injective function from $\vsub{H}{u}$ to a clique is an injective homomorphism. This means there are fewer injective homomorphisms from $\vsub{H}{u}$ into the common neighborhood of $U$ than into a clique of the same size. For any graphs $P$ and $Q$, the number of injective homomorphisms from $P$ to $Q$ is equal to the number of automorphisms of $P$ times the number of copies of $P$ in $Q$. Therefore there are also fewer copies of $\vsub{H}{u}$ in the common neighborhood of $U$ than there are in a clique of the same size, so as mentioned when discussing the conditions under which equality holds in \eqref{eq:count}, $G$ does not meet the bound.

\medskip

We are now prepared to show that $G$ meets the bound only if every component of $G$ containing a $u$-clique is a complete graph containing at least $t$ vertices. If $G$ is disconnected, we consider each component separately, and if a component $C$ does not contain a $u$-clique, we have seen $C$ does not contribute to either side of the bound and need not be considered. Therefore we may assume $G$ is connected and contains a $u$-clique. Either $u \ge 2$, in which case Claim 2 applies, or $u=1$ and $H$ is not a star. In this case, the graph that remains after removing any one dominating vertex contains an edge, so Claim 3 applies. In either case, any $u$-clique forms a clique with its common neighborhood. Applying Claim 1 to each $u$-set of this clique, no vertex outside of this clique can be adjacent to any vertex contained in the clique, so $G$ is a complete graph. In order for equality to hold in \eqref{eq:count}, there must be at least one extension $T$ of $U$, which requires this complete graph to have at least $t$ vertices.
\end{proof}

\subsubsection{Weighting \texorpdfstring{$t$}{t}-cliques by maximum star size} 

By taking $H=K_t$ and $u \in \set{1,2}$ in \cref{thm:local_star} we obtain the following corollaries. We note that Proposition 1 in \cite{MT22} is the case $t=2$ of \cref{cor:starcliquen}. 
\begin{cor}\label{cor:starcliquen}For every $n$-vertex graph $G$ and every clique size $t \ge 1$,
\[
     \sum_{T \in \K(G)} s^1_G(T) = \sum_{T \in \K(G)}\frac{1}{\binom{\theta_G(T)}{t-1}} \le \frac{n}{t}.
\]
When $t=2$, equality holds if and only if $G$ contains no isolated vertices and each component of $G$ is regular. When $t \ge 3$, equality holds if and only if $G$ is a disjoint union of complete graphs of order at least $t$. 
\end{cor}

\begin{proof}
Let $G$ be a graph on $n$ vertices. We set $H = K_t$ and $u = 1$ in \cref{thm:local_star}, so $s^1_G(T) = \frac{1}{\binom{\theta_G(T)}{t-1}}$. Any vertex $v$ of $K_t$ is a dominating vertex of $K_t$, so by \cref{thm:local_star} we have \[\sum_{T \in \K(G)} \frac{1}{\binom{\theta_G(T)}{t-1}} \le \frac{\scount{\vsub{K_t}{1}}{K_{t-1}}}{\binom{\dom(K_t)}{1}}\cdot \scount{K_{1}}{G} = \frac{n}{t}.\] 
 
When $t=2$, $H = K_2 = S_1$, so the equality condition matches the $u=1$ and $H = S_1$ case in \cref{thm:local_star}. When $t \ge 3$, the equality condition holds as cliques on $t \ge 3$ vertices are not stars and any non-empty component contains a $1$-clique, i.e. a vertex.
\end{proof}

Note that taking $t=1$ in \cref{cor:starcliquen} gives the true, if trivial, statement that an $n$-vertex graph $G$ has at most $n$ vertices, so equality holds for all graphs.

\begin{cor}\label{cor:starcliquem}For every $m$-edge graph $G$ and every clique size $t \ge 2$, 
\[
     \sum_{T \in \K(G)} s^2_G(T) = \sum_{T \in \K(G)}\frac{1}{\binom{\theta_G(T)-1}{t-2}} \le \frac{m}{\binom{t}{2}}.
\]
When $t \ge 3$, equality holds if and only if $G$ is a disjoint union of complete graphs of order at least $t$ and any number of isolated vertices.
\end{cor}

\begin{proof}
Let $G$ be a graph on $m$ edges. We set $H = K_t$ and $u = 2$ in \cref{thm:local_star}, so $s^2_G(T) = \frac{1}{\binom{\theta_G(T)-1}{t-2}}$. Any two vertices of $K_t$ are dominating vertices of $K_t$, so by \cref{thm:local_star} we have \[\sum_{T \in \K(G)} \frac{1}{\binom{\theta_G(T)-1}{t-2}} \le \frac{\scount{\vsub{K_t}{2}}{K_{t-2}}}{\binom{\dom(K_t)}{2}}\cdot \scount{K_{2}}{G} = \frac{m}{\binom{t}{2}}.\]

As above, the equality condition holds as cliques on $t \ge 3$ vertices are not stars. As edges are $2$-cliques, all other components must be isolated vertices.
\end{proof}

Similar to \cref{cor:starcliquen}, taking $t=2$ in \cref{cor:starcliquem} gives a true but trivial statement, namely than an $m$-edge graph has at most $m$ edges, so equality holds for all graphs.

\cref{cor:starcliquen} and \cref{cor:starcliquem} in turn imply the following two known theorems on the maximum number of $t$-cliques in bounded-degree graphs having a given number of vertices or a given number of edges, respectively. These theorems have also been proved using essentially the same argument as the one used in \cite{dW07} to give an upper bound on the total number of cliques of all sizes. The first theorem determines $\ex(n,K_t,S_r)$ asymptotically.

\begin{cor}[Wood~\cite{dW07}]\label{cor:woodn}
Let $t \ge 1$ and $G$ be a graph on $n$ vertices having $\Delta(G) \le r-1$. Then
\[ \scount{K_t}{G} \le \frac{n}{r}\binom{r}{t}. \]
For $t \ge 3$, equality holds if and only if $r \divides n$ and $G$ is a disjoint union of copies of $K_r$. 
\end{cor}

\begin{proof}
Let $G$ be a graph on $n$ vertices with maximum degree at most $r-1$. Then $\theta_G(T) \le r-1$ for every $T \in \K(G)$. By \cref{cor:starcliquen} we have \[
   \frac{\scount{K_t}{G}}{\binom{r-1}{t-1}} = \displaystyle\sum_{T \in \K(G)}\frac{1}{\binom{r-1}{t-1}} \le  \displaystyle\sum_{T \in \K(G)}\frac{1}{\binom{\theta_G(T)}{t-1}} \le \frac{n}{t},
\]
so $\scount{K_t}{G} \le \frac{n}{t}\binom{r-1}{t-1} = \frac{n}{r}\binom{r}{t}$. If $t \ge 3$ and equality holds, then every component is a clique. As $G$ has maximum degree at most $r-1$, these cliques have at most $r$ vertices. If some $G$ meeting the bound had a component $C$ with $a < r$ vertices, then $G-C$ would be a graph on $n-a$ vertices with $\frac{n}{r}\binom{r}{t} - \binom{a}{t}$ copies of $K_t$. As $t \ge 3$, we have
\[ a \binom{r}{t} = \frac{ar}{t} \binom{r-1}{t-1} > \frac{ar}{t} \binom{a-1}{t-1} = r\binom{a}{t} \]
and thus $\frac{a}{r}\binom{r}{t} > \binom{a}{t}$, so
\[ \frac{n}{r}\binom{r}{t} - \binom{a}{t} > \frac{n}{r}\binom{r}{t} - \frac{a}{r}\binom{r}{t} = \frac{n-a}{r} \binom{r}{t}\]
which contradicts the bound.
\end{proof}

If we instead fix the number of edges, we determine $\mex(m, K_t, S_r)$ asymptotically.

\begin{cor}[Wood~\cite{dW07}]\label{cor:woodm}
Let $t \ge 2$ and $G$ be a graph on $m$ edges with $\Delta(G) \le r-1$. Then
\[  \scount{K_t}{G} \le \frac{m}{\binom{r}{2}}\binom{r}{t}. \]
For $t \ge 3$, equality holds if and only if $\binom{r}{2} \divides m$ and $G$ is a disjoint union of copies of $K_r$ with any number of isolated vertices. 
\end{cor} 

\begin{proof}
Let $G$ be a graph on $m$ edges with maximum degree at most $r-1$. Again $\theta_G(T) \le r-1$ for every $T \in \K(G)$. By \cref{cor:starcliquem} we have \[
   \frac{\scount{K_t}{G}}{\binom{r-2}{t-2}} = \displaystyle\sum_{T \in \K(G)}\frac{1}{\binom{r-2}{t-2}} \le  \displaystyle\sum_{T \in \K(G)}\frac{1}{\binom{\theta_G(T)-1}{t-2}} \le \frac{m}{\binom{t}{2}},
\]
so $\scount{K_t}{G} \le \frac{m}{\binom{t}{2}}\binom{r-2}{t-2} = \frac{m}{\binom{r}{2}}\binom{r}{t}$. If $t \ge 3$ and equality holds, then every component is a clique. As $G$ has maximum degree at most $r-1$, these cliques have at most $r$ vertices. Suppose some $G$ meeting the bound has a component $C$ with $a < \binom{r}{2}$ edges. Then $a = \binom{x}{2}$ for some real number $x < r$, and $G-C$ is a graph on $m-a$ edges with at least $\frac{m}{\binom{r}{2}}\binom{r}{t} - \binom{x}{t}$ copies of $K_t$ by \cref{thm:LKKgraph}. As $t \ge 3$, we have
\[ a\binom{r}{t} = \binom{x}{2} \binom{r}{t} = \frac{x(x-1)r(r-1)}{2t(t-1)} \binom{r-2}{t-2} > \frac{x(x-1)r(r-1)}{2t(t-1)} \binom{x-2}{t-2} = \binom{r}{2}\binom{x}{t} \]
and thus $\frac{a}{\binom{r}{2}}\binom{r}{t} > \binom{x}{t}$, so
\[ \frac{m}{\binom{r}{2}}\binom{r}{t} - \binom{x}{t} > \frac{m}{\binom{r}{2}}\binom{r}{t} - \frac{a}{\binom{r}{2}}\binom{r}{t} = \frac{m-a}{\binom{r}{2}} \binom{r}{t}\]
which contradicts the bound.
\end{proof}

\subsubsection{Weighting copies of \texorpdfstring{$H$}{H} by maximum star size}

As mentioned, we can apply Theorem~\ref{thm:local_star} to a broader class of graphs $H$ than just cliques. This allows us to prove novel asymptotic results on $\ex(n,H,S_r)$
for any $H$ with at least one dominating vertex and on $\mex(n,H,S_r)$
for $H$ with at least two dominating vertices:

\begin{thm}\label{thm:H}
Let $H$ be a graph on $t$ vertices.
\begin{enumerate}[(i)]
\item If $H$ has at least one dominating vertex, then
\[ \ex(n,H,S_r) = (1-o(1))\scount{H}{\floor{ \tfrac{n}{r}} K_{r}}, \]
and
\item if $H$ has at least two dominating vertices, then
\[ \mex(m,H,S_r) = (1-o(1))\scount{H}{\floor[\Big]{ \tfrac{m}{\binom{r}{2}}} K_{r}}. \]
\end{enumerate}
\end{thm}

\begin{proof}
Let $G$ be an $S_r$-free graph on $n$ vertices and $m$ edges. Then for each $T \in \cH(G)$, we have $\theta_G(T) \le r-1$ and thus $s^1_G(T) \ge 1/\binom{r-1}{t-1}$ and $s^2_G(T) \ge 1/\binom{r-2}{t-2}$. Therefore, as long as $\dom(H) \ge 1$, let $v$ be a dominating vertex and apply Theorem~\ref{thm:local_star} with $u=1$ to get
\[ \scount{H}{G} \cdot \frac{1}{\binom{r-1}{t-1}} \le \sum_{T \in \cH(G)} s^1_G(T) \le \frac{\scount{\vsub{H}{1}}{K_{t-1}}}{\dom(H)} \cdot n \]
so that
\[ \scount{H}{G} \le n\binom{r-1}{t-1} \frac{\scount{\vsub{H}{1}}{K_{t-1}}}{\dom(H)} = (1-o(1))\scount{H}{\floor{\tfrac{n}{r}} K_{r}},\]
as, when $r \mid n$, we can count copies of $H$ in $\floor{\frac{n}{r}} K_r$ by first choosing a vertex $v$ of $G$ to act as a selected dominating vertex in $H$, then choosing $t-1$ of the $r-1$ other vertices in the same component, then choosing an embedding of $\vsub{H}{1}$ into those vertices (which is independent of the choice of $v$). Each copy of $H$ is counted this way once for each choice of selected dominating vertex in $H$. The asymptotic factor allows for $n$ not divisible by $r$.
Furthermore, as long as $\dom(H) \ge 2$, let $v$ and $w$ both be dominating vertices and apply Theorem~\ref{thm:local_star} with $u=2$ to get
\[ \scount{H}{G} \cdot \frac{1}{\binom{r-2}{t-2}} \le \sum_{T \in \cH(G)} s^2_G(T) \le \frac{\scount{\vsub{H}{2}}{K_{t-1}} }{\dom(H)} \cdot m \]
so that
\[ \scount{H}{G} \le m\binom{r-2}{t-2} \frac{\scount{\vsub{H}{2}}{K_{t-1}}}{\dom(H)} = (1-o(1))\scount{H}{\floor[\Big]{\tfrac{m}{\binom{r}{2}} }K_{r}},
\]
where similarly the asymptotic factor allows for $m$ not divisible by $\binom{r}{2}$.

In both cases we achieve a matching lower bound by taking as many disjoint copies of $K_{r}$ as possible and making the remaining vertices independent or making the remaining edges a matching. (For the values of $n$ and $m$ when we have some remaining vertices or edges, a better lower bound is given by forming a clique with the remaining vertices or a colex graph with the remaining edges.)
\end{proof}

Notice that cliques, stars, and all connected threshold graphs have at least one dominating vertex so are included as possible graphs $H$ in part $(i)$ of \cref{thm:H}.

\subsection{Hypergraphs}\label{subsec:hypergraphs}

We now consider localized bounds for hypergraphs of bounded degree. Recall that a hypergraph is \emph{$q$-uniform} if every edge is a set of $q$ vertices. The degree of a set of vertices $I$, denoted by $d(I)$, is the number of edges $E$ that contain $I$. Letting $i = \abs{I}$, the neighborhood of $I$ is the $(q-i)$-uniform hypergraph $\set{E\setminus I: I \subset E \in E(\cH)}$. For a $q$-uniform hypergraph $\cH$ and $1 \le i < q$, we write $\Delta_i(\cH)$ for the maximum degree $d(I)$ over all sets $I$ of $i$ vertices.

For $t \ge q$, we denote by $K^{(q)}_t$ the complete $q$-uniform hypergraph on $t$ vertices. We write $\K(\cH)$ for the set of $t$-cliques in $\cH$, i.e., $\K(\cH) = \set{S \subseteq V(\cH) : \cH[S] \cong K^{(q)}_t}$. We use the following upper bound on $\scount{K^{(q)}_t}{\cH}$, which is proved in \cite[Theorem 32]{KR22} as an immediate consequence of Lov\'{a}sz' approximate version of the Kruskal-Katona theorem.
\begin{thm}[Lov\'{a}sz \cite{L79}]\label{thm:LKK}
	Let $q, t \in \N$ with $t \ge q$. Let $\cH$ be a $q$-uniform hypergraph. Write the number of edges of $\cH$ in the form $\binom{x}{q}$, where $x \ge q-1$ is real. Then $\scount{K^{(q)}_t}{\cH} \le \binom{x}{t}$.
\end{thm}

The following theorem generalizes \cref{cor:starcliquen} to $q$-uniform hypergraphs. Note that when $q=2$, we have $i=1$ and $x(I) = d(I) + i = d(v)+1$ for $I = \set{v}$, and so the function $s(T)$ in the following theorem can be thought of as an extension of the function $s_G^1(T)$ of Theorem~\ref{thm:local_star} to hypergraphs.

\begin{thm} \label{thm:local_kr22}
Let $t \ge q > i \ge 1$ and suppose $\cH$ is a $q$-uniform hypergraph on $n$ vertices. For each $I \in \binom{V(\cH)}{i}$, define $x(I) \ge q-i-1$ by the equation $d(I) = \binom{x(I)-i}{q-i}$, and, for each $T \in \K(\cH)$, define
\[ x(T) = \max\set[\Big]{x(I) : I \in \binom{T}{i}} \quad \text{and} \quad s(T) = \frac{1}{\binom{x(T)-i}{t-i}	}. \]
Then $s(T)$ is well-defined and decreasing as a function of $x(T)$, 
\[\sum_{T \in \K(\cH)}s(T) \le \frac{\binom{n}{i}}{\binom{t}{i}}, \]
and there is an infinite family of hypergraphs that achieve the bound.
\end{thm}

\begin{proof}
Let $I \in \binom{V(\cH)}{i}$. For every $T \in \K(I)$, we have $x(T) \ge x(I)$ by definition. If $\K(I)$ is nonempty, then $d(I) \ge \binom{t-i}{q-i}$ and $x(I) \ge t$. Therefore every $T \in \K(\cH)$ has $x(T) \ge t$, so $w(T)$ is a decreasing function of $x(T)$. Hence $T \in \K(I)$ implies $w(T) \le 1/\binom{x(I)-i}{t-i}$. Therefore 
\begin{align*}
\binom{t}{i}\sum_{T \in \K(\cH)}s(T) &= \sum_{I \in \binom{V(\cH)}{i}} \sum_{T \in \K(I)}s(T) \\
&\le \sum_{I \in \binom{V(\cH)}{i}} \sum_{T \in \K(I)}\frac{1}{\binom{x(I)-i}{t-i}} \\
&\le \sum_{I \in \binom{V(\cH)}{i}} \frac{\binom{x(I)-i}{t-i}}{\binom{x(I)-i}{t-i}} \\
&= \binom{n}{i},
\end{align*}where the second inequality follows from applying \cref{thm:LKK} to the neighborhood of $I$.

Design theory provides an infinite family of graphs that meet this bound; we direct the reader to \cite{KR22} for more information on such hypergraphs. If $\cH$ is a $q$-shadow of a Steiner system $S(i,r,n)$ for some $r$ then by \cite[Lemma 38(b)]{KR22} we have $x(I)=r$ for every $I$ and $x(T)=r$ for every $T$, so $s(T) = \frac{1}{\binom{r-i}{t-i}}$. By \cite[Lemma 38(a)]{KR22} we have $\scount{K^{(q)}_t}{\cH} = \binom{r}{t}\frac{\binom{n}{i}}{\binom{r}{i}}$. Therefore \[\sum_{T \in \K(\cH)}s(T) = \frac{\binom{r}{t}\binom{n}{i}}{\binom{r}{i}\binom{r-i}{t-i}}=\frac{\binom{n}{i}}{\binom{t}{i}}.\qedhere\]
 \end{proof}

It seems interesting and challenging to characterize all of the extremal $q$-graphs in \cref{thm:local_kr22}. See \cite[Theorem 43]{KR22} for a related characterization of the extremal $q$-graphs in the non-localized theorem.

As a corollary of \cref{thm:local_kr22} we obtain the following theorem of Radcliffe and the first author on maximizing the number of $t$-cliques among bounded-degree $q$-uniform hypergraphs.

\begin{thm}[Kirsch and Radcliffe~\cite{KR22}] \label{thm:kr22}
Let $1 \le i < q \le t$ and suppose $\cH$ is an $q$-uniform hypergraph on $n$ vertices such that $\Delta_i(\cH) \le \binom{x-i}{q-i}$ for some real number $x \ge q$. Then
\[ \scount{K^{(q)}_t}{\cH} \le \frac{\binom{n}{i}}{\binom{x}{i}}\binom{x}{t}. \]
\end{thm} 

\begin{proof}[Proof using \cref{thm:local_kr22}] The condition $\Delta_i(\cH) \le \binom{x-i}{q-i}$ implies that $x(I) \le x$ for every $I \in \binom{V(\cH)}{i}$, so $x(T) = \max\set{x(I) : I \in \binom{T}{i}} \le x$ for every $T \in \K(\cH)$. \cref{thm:local_kr22} gives
\[
\frac{\scount{K^{(q)}_t}{\cH}}{\binom{x-i}{t-i}}=\sum_{T\in\K(\cH)}\frac{1}{\binom{x-i}{t-i}} \le \sum_{T\in\K(\cH)}w(T) \le \frac{\binom{n}{i}}{\binom{t}{i}},
\]
so $\scount{K^{(q)}_t}{\cH} \le \frac{\binom{n}{i}}{\binom{t}{i}}\binom{x-i}{t-i} = \frac{\binom{n}{i}}{\binom{x}{i}}\binom{x}{t}$.
\end{proof}

\section{Open Problems} \label{sec:conclusion}

We briefly mention a few additional instances of problems that we believe are amenable to localized extensions.

The following conjecture is a localized form of a theorem of Frohmader \cite{F08}, as phrased in \cite[Theorem 8]{KR22}, on maximizing the number of $t$-cliques among $m$-edge, $K_{r+1}$-free graphs. 
\begin{conj}\label{conj:frohmader}
 Let $t \ge 2$. For each $T \in \K(G)$, define
 \[ \alpha_G(T) = \max\set{k : T \subseteq V(S) \textnormal{ for some } S \subseteq G \st S \cong K_k} \quad \text{and} \quad w'_G(T) = \frac{\binom{\alpha_G(T)}{2}^{t/2}}{\binom{\alpha_G(T)}{t}}. \]
 For every $m$-edge graph $G$, 
	\[
		\sum_{T \in \K(G)} w'_G(T) \le m^{t/2}.
	\]
\end{conj}
\noindent After a preprint of this paper was made available, Arag{\~a}o and Souza \cite{AS23} announced a proof of a generalization of \cref{conj:frohmader}.

Many extremal results on paths, beginning with the results of Erd\H{o}s and Gallai~\cite{EG59}, are consequences of extremal theorems regarding cycles. While the family of cycle graphs $\set{C_3,C_4,\ldots}$ does not have the subgraph inclusion property shared by cliques, paths, and stars, these results consider graphs of bounded circumference (that is, maximum cycle length). The techniques in this paper often bounded a weight function by arguing a maximal structure could not be extended; cycles do not allow such arguments, which could make proving localized results more difficult. Nevertheless, we provide the following weight function and conjectures based on results of Luo~\cite{rL18} and Chakraborti and Chen~\cite{CC21}, respectively.

\begin{defn}
Let $t \ge 2$. For each $T \in \K(G)$, define
    \[ \gamma_G(T) = \max\set{k : T \subseteq V(S) \textnormal{ for some } S \subseteq G \st S \cong C_{k}}. \]
\end{defn}

\begin{conj}
Let $t \ge 2$. For each $T \in \K(G)$, define
\[ c_G(T) = \frac{\gamma_G(T)-1}{\binom{\gamma_G(T)}{t}}. \]
Then $c_G(T)$ is well-defined and decreasing in $\gamma_G(T)$, and
\[ \sum_{T \in \K(G)} c_G(T) \le n-1, \]
with equality if and only if each $2$-connected component of $G$ is a complete graph of order at least $t$.
\end{conj}

\begin{conj}
Let $t \ge 2$. For each $T \in \K(G)$, define
\[ c'_G(T) = \frac{\binom{\gamma_G(T)}{2}}{\binom{\gamma_G(T)}{t}}. \]
Then $c'_G(T)$ is well-defined and decreasing in $\gamma_G(T)$, and
\[ \sum_{T \in \K(G)} c'_G(T) \le m, \]
with equality if and only if each $2$-connected component of $G$ is a complete graph of order at least $t$ and any number of isolated vertices.
\end{conj}

It may be possible to generalize \cref{cor:starcliquem} to hypergraphs. We make the following conjecture as a localized version of Theorem 51 in \cite{KR22}, analogously to \cref{thm:local_kr22}.
\begin{conj}Let $t \ge q > i \ge 1$ and suppose $\cH$ is a $q$-uniform hypergraph on $m$ edges. For each $I \in \binom{V(\cH)}{i}$, define $x(I) \ge q-i-1$ by the equation $d(I) = \binom{x(I)-i}{q-i}$, and, for each $T \in \K(\cH)$, define
\[ x(T) = \max\set[\Big]{x(I) : I \in \binom{T}{i}} \quad \text{and} \quad s'(T) = \frac{1}{\binom{x(T)-q}{t-q}	}. \]
Then
\[\sum_{T \in \K(\cH)}s'(T) \le \frac{m}{\binom{t}{q}}. \]
\end{conj}

Finally, we used \cref{thm:local_kr22} to obtain new asymptotically tight bounds on $\ex(n,H,S_r)$ when $H$ has at least one dominating vertex and on $\mex(m,H,S_r)$ when $H$ has at least two dominating vertices. It may be possible to prove similar results for hypergraphs.

\begin{question}Can \cref{thm:H} (or \cref{thm:local_star}) be generalized to the setting of $q$-uniform hypergraphs with bounded maximum $i$-degree, perhaps with $i=1$ or $i=q-1$, in such a way as to obtain new generalized Tur\'{a}n-type results for hypergraphs?
\end{question}

\section*{Acknowledgments} The authors thank Jamie Radcliffe for valuable discussions.

\bibliographystyle{amsplain}
\bibliography{local_extremal}

\providecommand{\bysame}{\leavevmode\hbox to3em{\hrulefill}\thinspace}
\providecommand{\MR}{\relax\ifhmode\unskip\space\fi MR }
\providecommand{\MRhref}[2]{%
  \href{http://www.ams.org/mathscinet-getitem?mr=#1}{#2}
}
\providecommand{\href}[2]{#2}
\begin{thebibliography}{10}

\bibitem{AS16}
Noga Alon and Clara Shikhelman, \emph{Many {$T$} copies in {$H$}-free graphs},
  J. Combin. Theory Ser. B \textbf{121} (2016), 146--172.

\bibitem{AS23}
Lucas {Arag{\~a}o} and Victor {Souza}, \emph{{Localised graph Maclaurin
  inequalities}}, arXiv e-prints (2023), arXiv:2301.13189.

\bibitem{BBL23}
J{\'o}zsef {Balogh}, Domagoj {Brada{\v{c}}}, and Bernard {Lidick{\'y}},
  \emph{{Weighted Tur{\'a}n theorems with applications to Ramsey-Tur{\'a}n type
  of problems}}, arXiv e-prints (2023), arXiv:2302.07859.

\bibitem{BCMM22}
J{\'o}zsef {Balogh}, Ce~{Chen}, Grace {McCourt}, and Cassie {Murley},
  \emph{{Ramsey-Tur{\'a}n Problems with small independence numbers}}, arXiv
  e-prints (2022), arXiv:2207.10545.

\bibitem{bB12}
B{\'e}la Bollob{\'a}s, \emph{Graph theory: an introductory course}, vol.~63,
  Springer Science \& Business Media, 2012.

\bibitem{BG08}
B{\'e}la Bollob{\'a}s and Ervin Gy{\H{o}}ri, \emph{Pentagons vs. triangles},
  Discrete Mathematics \textbf{308} (2008), no.~19, 4332--4336.

\bibitem{dB22}
Domagoj {Brada{\v{c}}}, \emph{{A generalization of Tur{\'a}n's theorem}}, arXiv
  e-prints (2022), arXiv:2205.08923.

\bibitem{CC20}
Debsoumya {Chakraborti} and Da~Qi {Chen}, \emph{{Exact results on generalized
  Erd{\H{o}}s-Gallai problems}}, arXiv e-prints (2020), arXiv:2006.04681.

\bibitem{CC21}
Debsoumya Chakraborti and Da~Qi Chen, \emph{Many cliques with few edges and
  bounded maximum degree}, Journal of Combinatorial Theory, Series B
  \textbf{151} (2021), 1--20.

\bibitem{rD17}
Reinhard Diestel, \emph{Extremal graph theory}, Graph theory, Springer, 2017,
  pp.~173--207.

\bibitem{pE62}
Paul Erd{\H{o}}s, \emph{On the number of complete subgraphs contained in
  certain graphs}, Magyar Tud. Akad. Mat. Kutat{\'o} Int. K{\"o}zl \textbf{7}
  (1962), no.~3, 459--464.

\bibitem{EG59}
Paul Erd{\H{o}}s and Tibor Gallai, \emph{On maximal paths and circuits of
  graphs}, Acta Mathematica Hungarica \textbf{10} (1959), no.~3-4, 337--356.

\bibitem{F08}
Andrew Frohmader, \emph{Face vectors of flag complexes}, Israel Journal of
  Mathematics \textbf{164} (2008), no.~1, 153--164.

\bibitem{dG22}
D{\'a}niel {Gerbner}, \emph{{Paths are Tur{\'a}n-good}}, arXiv e-prints (2022),
  arXiv:2204.07638.

\bibitem{GP22}
D{\'a}niel Gerbner and Cory Palmer, \emph{Some exact results for generalized
  tur{\'a}n problems}, European Journal of Combinatorics \textbf{103} (2022),
  103519.

\bibitem{GS18}
Lior Gishboliner and Asaf Shapira, \emph{A generalized tur{\'a}n problem and
  its applications}, Proceedings of the 50th Annual ACM SIGACT Symposium on
  Theory of Computing, 2018, pp.~760--772.

\bibitem{GPS91}
Ervin Gy{\"o}ri, J{\'a}nos Pach, and Mikl{\'o}s Simonovits, \emph{On the
  maximal number of certain subgraphs in {$K_r$}-free graphs}, Graphs and
  Combinatorics \textbf{7} (1991), no.~1, 31--37.

\bibitem{HHKNR13}
Hamed Hatami, Jan Hladk{\'y}, Daniel Kr{\'a}l', Serguei Norine, and Alexander
  Razborov, \emph{On the number of pentagons in triangle-free graphs}, Journal
  of Combinatorial Theory, Series A \textbf{120} (2013), no.~3, 722--732.

\bibitem{KR22}
Rachel {Kirsch} and Jamie {Radcliffe}, \emph{{Many cliques in bounded-degree
  hypergraphs}}, arXiv e-prints (2022), arXiv:2207.02336.

\bibitem{L79}
L.~Lov\'{a}sz, \emph{Combinatorial problems and exercises}, North-Holland
  Publishing Co., Amsterdam-New York, 1979.

\bibitem{rL18}
Ruth Luo, \emph{The maximum number of cliques in graphs without long cycles},
  Journal of Combinatorial Theory, Series B \textbf{128} (2018), 219--226.

\bibitem{MT22}
David {Malec} and Casey {Tompkins}, \emph{{Localized versions of extremal
  problems}}, arXiv e-prints (2022), arXiv:2205.12246.

\bibitem{MNNRW22}
Natasha {Morrison}, JD~{Nir}, Sergey {Norin}, Pawe{\l} {Rz{a}{\.z}ewski}, and
  Alexandra {Wesolek}, \emph{{Every graph is eventually Tur{\'a}n-good}}, arXiv
  e-prints (2022), arXiv:2208.08499.

\bibitem{RU18}
Jamie {Radcliffe} and Andrew {Uzzell}, \emph{{Stability and Erd{\H{o}}s--Stone
  type results for $F$-free graphs with a fixed number of edges}}, arXiv
  e-prints (2018), arXiv:1810.04746.

\bibitem{pT41}
Paul Tur\'an, \emph{Eine {E}xtremalaufgabe aus der {G}raphentheorie}, Mat. Fiz.
  Lapok \textbf{48} (1941), 436--452.

\bibitem{dW07}
David~R Wood, \emph{On the maximum number of cliques in a graph}, Graphs and
  Combinatorics \textbf{23} (2007), no.~3, 337--352.

\bibitem{aZ49}
Alexander~Aleksandrovich Zykov, \emph{On some properties of linear complexes},
  Matematicheskii sbornik \textbf{66} (1949), no.~2, 163--188.

\end{thebibliography}

\end{document}